\documentclass[12pt,a4paper,english]{smfart}
\usepackage[english]{babel}
\usepackage{ragged2e}
\usepackage{smfthm, mathabx}
\usepackage{stmaryrd}
\usepackage{amssymb}
\usepackage{amsmath}
\usepackage{amsthm}
\usepackage{amsfonts}
\usepackage{graphicx}
\usepackage{enumerate}
\usepackage{comment}

\usepackage{tikz-cd}

\usepackage{appendix}

\usepackage{euscript,mathrsfs}
\usepackage[utf8]{inputenc} 
\usepackage{longtable}
\usepackage{dsfont}

\usepackage[OT2,T1]{fontenc}

\usepackage[all,cmtip]{xy}
\usepackage{alltt}

\usepackage{calrsfs}

\makeatletter
\def\thm@space@setup{%
  \thm@preskip=\parskip \thm@postskip=0pt
}
\makeatother

\xyoption{all}

\usepackage{float}

\usepackage{fullpage}

\usepackage{color}

\newtheorem{Theorem}{Theorem}
\newtheorem{prp}{Proposition}
\newtheorem{Exa}{Example}
\newtheorem*{thm}{Theorem}

\author[Oussama Hamza]{Oussama Hamza}
\address{Department of Mathematics, Western University, London, Ontario, Canada N6A5B7}
\email{ohamza3@uwo.ca}

\title{On extensions of number fields with given quadratic algebras and cohomology}

\thanks{I am grateful to Christian Maire, J{\'a}n Min{\'a}{\v c} and Thomas Weigel for careful review of this work. I am also thankful to Tatyana Barron, Elyes Boughattas, Mathieu Florence, Chris Hall, John Labute and Donghyeok Lim for discussions and encouragements.
Let me also thank Shubhankar Bhatt, Michal Cizek, Jarl G. Flaten Taxer\aa s, Tao Gong, Michael Rogelstad and Sayantan Roy Chowdhury for their interests in this work. Finally, I praise the anonymous referees for provided comments and careful readings.}

\newcommand{\Q}{\mathbb{Q}}
\newcommand{\F}{\mathbb{F}}
\newcommand{\Z}{\mathbb{Z}}
\newcommand{\NN}{\mathbb{N}}

\def\grad{{\rm Grad}}

\def\Gal{{\rm Gal}}
\def\cd{{\rm cd}}

\def\bA{{\mathbf{A}}}
\def\bB{{\mathbf{B}}}

\def\bE{{\mathbf{E}}}
\def\bN{{\mathbf{N}}}

\def\O{{\mathcal O}}

\def\PP{{\mathcal P}}
\def\J{{\mathcal J}}
\def\E{{\mathcal E}}

\def\A{{\mathcal A}}
\def\AA{{\mathbb A}}

\def\ext{{\rm Ext}}

\def\I{{\mathcal I}}
\def\bi{{\mathbf i}}
\def\br{{\mathbf r}}
\def\bz{{\mathbf z}}
\def\bj{{\mathbf j}}
\def\bk{{\mathbf k}}
\def\bl{{\mathbf l}}
\def\bu{{\mathbf u}}
\def\bv{{\mathbf v}}
\def\br{{\mathbf r}}
\def\bz{{\mathbf z}}
\def\bo{{\mathbf o}}
\def\bq{{\mathbf q}}

\def\B{{\mathcal B}}
\def\CC{{\mathcal C}}

\def\L{{\rm L}}

\begin{document}

\begin{abstract}
We introduce a criterion on the presentation of finitely presented pro-$p$ groups which allows us to compute their cohomology groups and infer non-analytic quotients, with cohomological dimension strictly larger than two, from (non-free) mild groups. 
We interpret these groups as Galois groups over~$p$-rational fields with prescribed ramification and splitting.
\end{abstract}

\maketitle

\section*{Introduction}

Presentations of (pro-$p$) groups via generators and relations have played an important role in the development of group theory (see \cite[Chapter~$2$]{lyndon1977combinatorial} and \cite{magnus2004combinatorial}), and more generally in the current theory of profinite groups and especially pro-$p$ groups. These methods, combined with cohomological results, are also used to detect Galois groups of~$p$-extensions, see for instance Koch~\cite{Koch}, Min{\'a}{\v{c}}-Rogelstad-T{\^a}n~\cite{minavc2016triple} and~\cite{minavc2020relations}, and Efrat-Quadrelli~\cite{efrat2019kummerian}. 

Shafarevich's great idea and insight was to present important Galois pro-$p$ groups via generators and relations, and to search for a numerical criterium, depending on the presentation, for proving that some of these groups are infinite. 
In the work with Golod~\cite{golod1964class}, he succeeded to make this idea precise, using associated filtrations and graded algebras techniques. Their numerical criterium was later refined to a famous Golod-Shafarevich criterium: if a pro-$p$ group admits~$d$ generators and~$r$ relations satisfying~$d^2\geq 4r$, then it is infinite (see \cite[Chapter IX]{CA}).

Around the same time, Lazard also inferred "l'Alternative des Gocha" (from the names of Golod and Shafarevich, see \cite[Appendice~$A.3$, Théorème~$3.11$]{LAZ}) which characterizes the topological structure of a pro-$p$ group from the gradation of its group algebra. In the early~$21^{\rm st}$ century, Labute-Min{\'a}{\v c}~\cite{Labute} and~\cite{Labute-Minac}, and Forré~\cite{FOR} used Anick's techniques~\cite{AN0} to define mild groups and infer FAB groups, i.e. groups such that every open subgroup has finite abelianization, of cohomological dimension~$2$.

In this paper, we construct non-analytic quotients of non-free mild groups of larger cohomological dimension by using and enriching previous techniques: presentations of pro-$p$ groups, projective resolutions, graded algebras, graph theory and Gröbner basis. From the "cutting tower" strategy introduced by Hajir-Maire-Ramakrishna~\cite{hajir2021cutting}, we conclude this article with arithmetical examples (see Theorem~\ref{prescribed cohomological dimension} below).

\subsection*{Arithmetic context}
Let~$p$ be a prime number and~$K$ be a~$p$-rational number field. The latter means that the Galois group~$G_K$, of the maximal pro-$p$ extension of~$K$ unramified outside~$p$, is isomorphic to a finitely generated free pro-$p$ group (see \cite[Théorème et Définition $2.1$]{nguyen1990arithmetique} for other equivalent properties). By a conjecture of Gras \cite[Conjecture~$8.11$]{gras2016theta}, if~$K$ is a fixed number field, then it is~$p$-rational for every prime~$p$ large enough.

Introduce~$T$ a finite set of finite primes of~$K$. Denote by~$G_K^T$ the Galois group of the maximal pro-$p$ extension of~$K$ unramified outside~$p$ and totally splitting in~$T$. We infer a free presentation~$G_K^T=G_K/R$, with~$R$ a normal closed subgroup of~$G_K$ presented by relations~$\{l_i\}_{i\in |T|}$. From the "cutting tower" strategy (see~\cite[Part~$2$]{split} or~\cite{hajir2021cutting})  based on the Chebotarev density Theorem, one can choose a set of primes~$T$ in~$K$ such that~$G_K^T$ has a \textit{mild} presentation (see \cite[Definition $1.1$]{Labute}), so cohomological dimension~$2$.  Mild groups play an important role in the understanding of Galois extensions with prescribed ramification and splitting (see \cite{Labute}, \cite{Labute-Minac} and \cite{schmidt2008uber}). 

Using the theory of Right Angled Artin Groups (RAAGs, see for instance \cite{bartholdi2020right}, \cite{wade2011lower} and \cite[Part~$2$]{lorensen2010groups}), we can construct quotients of~$G_K$ with prescribed cohomology. Let us fix~$\{x_1,\dots,x_d\}$ a minimal set of generators of~$G_K$ and an undirected graph~$\Gamma$ with set of vertices~$[\![1;d]\!]$. We define~$G(\Gamma)$ as a quotient of~$G_K$ by commutators~$[x_i;x_j]$ whenever~$\{i,j\}$ is an edge of~$\Gamma$. The dimension of the~$n$-th cohomology group of~$G(\Gamma)$ is given by~$c_n(\Gamma)$, the number of~$n$-cliques of~$\Gamma$: i.e.\ complete subgraphs of~$\Gamma$ with~$n$ vertices. 





In this work, we investigate quotients of mild groups with large finite cohomological dimension, using ideas introduced by RAAGs, and we realise them as absolute Galois groups with prescribed ramification and splitting.
Let~$G$ be a quotient of~$G_K$ and set~$h^n(G)$ to be the dimension of~$H^n(G;\F_p)$. 
We prove the following result:

  \begin{Theorem}\label{prescribed cohomological dimension}
Let~$\Gamma:=\Gamma_{\bA}\sqcup \Gamma_{\bB}$ be a graph where~$\Gamma_{\bA}$ is bipartite. 
Then, there exist a totally imaginary field~$K$ and a set~$T$ of primes in~$K$ such that~$G_K^T$ is (minimally) presented by generators $\{x_1;\dots;x_d\}$ indexed on the vertices of $\Gamma$ and  relations~$l_{\bA}:=\{l_{ij}; (i,j)\in \bA\}$ which, modulo the third Zassenhaus filtration of~$G_K$, satisfy the equality~$l_{ij}\equiv [x_i;x_j]$. In particular~$G_K^T$ is mild and non-free.
 
Furthermore, there exists a non-analytic quotient~$G$ of~$G_K^{T}$, such that for~$n\geq 2$, $h^n(G)=c_n(\Gamma)$. Consequently the cohomological dimension of~$G$ is equal to~$\max(2;n_{\Gamma_{\bB}})$, with~$n_{\Gamma_{\bB}}$ the clique number of~$\Gamma_{\bB}$. 
\end{Theorem}

Class Field Theory allows us to construct composita of $\Z_p$-extensions on~$K$ (see for instance \cite[Proposition $11.1.1$ and Theorem $11.1.2$]{NSW}). Alternatively, Schmidt's results also enable us to construct (non-analytic) quotients of absolute Galois groups with prescribed ramification which are mild. By applying \cite[Theorem $1.1$]{schmidt2008uber}, there exists a set of places $S$ of $K$ containing all places above~$p$ such that the absolute Galois group of $K$ unramified outside $S$ is mild and non-free. Then RAAG Theory provides us with extensions unramified outside~$p$ of cohomological dimension between three and the number of generators of~$G_K$. 

Theorem~\ref{prescribed cohomological dimension} strengthens the previous observation by constructing non-analytic extensions of $K$ that additionally totally splits in~$T$ and have cohomological dimensions close to the number of generators of~$G_K$ (see Corollary~\ref{final coro}). 

\subsection*{Cohomological results}
We first describe our main objects of study.

Let us denote by~$G$ a finitely presented pro-$p$ group with presentation~$G=F/R$, where~$F$ is a free pro-$p$ group with generators~$\{x_1,\dots,x_d\}$, and~$R$ is a normal closed subgroup of~$F$ generated by a finite family~$\{l_1,\dots, l_r\}$. We define~$E(G)$ as the completed group algebra of~$G$ over~$\F_p$. This is an augmented algebra, and we denote by~$E_n(G)$ the~$n$-th power of the augmentation ideal of~$E(G)$. Introduce 
  \begin{equation*}
  \begin{aligned}
\E_n(G)& :=  E_n(G)/E_{n+1}(G), \quad &\text{and}& \quad & \E(G) :=& \bigoplus_{n\in \NN}\E_n(G).
\end{aligned}
\end{equation*}

The graded algebra~$\E(G)$ plays a fundamental role in this article, and more generally in the understanding of filtrations (see \cite[Chapitre II and Appendice~$A.3$]{LAZ}, \cite{Labute}, \cite{Minac} and \cite{hamza2022isotypical}), topology (see \cite[Alternative des Gocha, Théorème~$3.11$, Appendice~$A.3$]{LAZ}) and cohomology (see \cite{Labute}, \cite{Labute-Minac}, \cite{minac2021koszul}, \cite{minavc2022mild}) of~$G$. Note that~$H^n(G;\F_p)$ is a discrete~$\F_p$-vector space, and denote by~$H^\bullet(G)$ the graded algebra~$\bigoplus_n H^n(G;\F_p)$ with product given by cup-product. We emphazise links between~$E(G)$,~$\E(G)$ and~$H^\bullet(G)$.

In \cite{BRU}, Brumer defined the functor~$\ext$ for compact modules, and showed that (\cite[Lemma~$4.2$]{BRU} and \cite[Part~$3.9$]{Koch}) we have an isomorphism of graded algebras~$H^\bullet(G)\simeq \ext^\bullet_{E(G)}(\F_p;\F_p)$, where the product is given by the cup-product. Furthermore, using May spectral sequence (see \cite[Theorem~$5.1.12$]{weigelspectralsequence}), we obtain an identification of~$H^\bullet(G)$ and~$\ext_{\E(G)}^\bullet(\F_p;\F_p)$ when~$\E(G)$ is Koszul, i.e. the trivial~$\E(G)$-module~$\F_p$ admits a free weighted-$\E(G)$ resolution~$(\PP_\bullet;\delta_\bullet)$, where~$\PP_i$ is generated by elements of degree~$i$ (we refer to \cite[Chapter~$2$]{polishchuk2005quadratic} for further references on Koszul algebra):
  \begin{prp}\label{compcoho}
If~$\E(G)$ is a Koszul algebra, then we have the following isomorphism of graded algebras:
$$H^\bullet(G)\simeq \ext_{\E(G)}^\bullet(\F_p;\F_p),$$
where the product is given by the cup-product. The algebra~$H^\bullet(G)$ is the quadratic dual of~$\E(G)$.
\end{prp}

Min{\'a}{\v{c}}-Pasini-Quadrelli-Tân already observed, in \cite[Proof of Theorem~$4.6$]{minavc2022mild}, that if~$G$ admits a presentation with quadratic relation, i.e.\ $l\subset F_2\setminus F_3$, which is additionally assumed to be mild, then~$\E(G)$ is Koszul. They also observed that if~$G$ is mild and~$H^\bullet(G)$ is quadratic, then~$H^\bullet(G)$ is the quadratic dual of~$\E(G)$. As a direct consequence of Proposition~\ref{compcoho}, we complete \cite[Theorem~$1.3$]{minavc2022mild}: if the group~$G$ admits a mild presentation with quadratic relations, then~$H^\bullet(G)$ is the quadratic dual of the Koszul algebra~$\E(G)$. 
For more details on quadratic duals, we refer to \cite[Part~$1.2$]{polishchuk2005quadratic}.



\subsection*{Computation of graded algebras}
Currently, the algebra~$\E(G)$ is only known when~$G$ is free, or mild, or in a few other specific cases (see \cite{Labute}, \cite{Labute-Minac} and \cite{Minac}). We give a criterion on the presentation of~$G$ which allows us to compute~$\E(G)$. As a consequence, we obtain the cohomology groups of a pro-$p$ group~$G$ (which will be here a quotient of a mild group) directly from its presentation. We are mostly inspired by the theory of RAAGs (see for instance \cite{bartholdi2020right} and \cite{wade2011lower}) and the work of Koch~\cite{koch1977pro} and Forré~\cite[Theorem~$3.7$]{FOR}. Let us now explain the strategy we adopt in this article to construct situations where~$\E(G)$ is Koszul. 

The Magnus isomorphism from~\cite[Chapitre II, Partie~$3$]{LAZ} gives us a surjection, that we denote by~$\phi$, between~$E(G)$ (resp.\ $\E(G)$) and the~$\F_p$-algebra of noncommutative series (resp.\ polynomials) over a set of variables~$\mathbf{X}:=\{X_1,\dots,X_d\}$, that we denote by~$E$ (resp.\ $\E$). In particular~$\E(G)$ is a quotient of~$\E$, and we denote by~$\I$ its kernel. It is in general difficult to explicitly compute the ideal~$\I$.

From the Magnus isomorphism, we write~$w_i:=\phi(l_i-1)$ as a sum of homogeneous polynomials in~$E$. A priori, every homogeneous polynomial in~$w_i$ plays a role in the computation of the ideal~$\I$. Labute~\cite{Labute} and Forré~\cite{FOR}, following ideas of Anick~\cite{AN0}, gave a criterion (mild presentation) on the presentation of~$G$ such that the ideal~$\I$ is generated only by the dominant term of~$w_i$. However, this criterion restricts the cohomological dimension of~$G$ to less than or equal to two. In this paper, we give another criterion, ensuring that~$\I$ is also generated by dominant terms of~$w_i$ and in addition to the mild case, we infer situations where the cohomological dimension is strictly larger than two.

Let~$\Gamma:=(\bN, \bE)$ be a graph with set of vertices~$\bN:=[\![1;d]\!]$ and set of edges~$\bE$. 
We introduce a set~$l_{\bE}:=\{l_{ij}\}_{\{i,j\} \in \bE}$ of relations in~$F$, and we state the following condition on the graph~$\Gamma$ and the family~$l_{\bE}$:
  \begin{equation}\label{condgradgroup}
\left\{   \begin{aligned}
\bullet &\text{The graph }\Gamma \text{ can be written as a disjoint union of two components }\\
& \text{that we call } \Gamma_{\bA} \text{ and } \Gamma_{\bB}, \text{ with sets of edges~$\bA$ and~$\bB$}.\\
\bullet & \text{The graph } \Gamma_{\bA}\text{ is bipartite, and } 
\\ & w_{ij}:=\phi(l_{ij}-1) \equiv [X_i;X_j] \pmod{E_3},\text{ for~$\{i,j\}\in \bA$}.\\
\bullet &\text{We have }  l_{uv}=[x_u;x_v], \text{ for~$\{u,v\}\in \bB$}.
\end{aligned} \right.
\end{equation}
Let us call~$\I(\Gamma)$ the ideal in~$\E$ generated by the family~$\{[X_i;X_j]\}_{\{i,j\}\in \bE}$, the grad of $w_{ij}$ when~$l_{\bE}$ satisfies the Condition \eqref{condgradgroup}, and call~$\E(\Gamma)$ the graded algebra~$\E(\Gamma):=\E/\I(\Gamma)$. We use ideas from Forré~\cite{FOR}, Wade~\cite{wade2011lower}, Labute-Min{\'a}{\v{c}}~\cite{Labute} and~\cite{Labute-Minac}, Min{\'a}{\v{c}}-Pasini-Quadrelli-T{\^a}n~\cite{minac2021koszul} and~\cite{minavc2022mild}, Anick~\cite{AN1} and Ufnarovskij~\cite{Ufnarovskij1995CombinatorialAA} to show that if~$G$ admits a presentation satisfying the Condition \eqref{condgradgroup}, we have~$\I=\I(\Gamma)$. Then we infer:
  \begin{Theorem}\label{gradgroup}
Assume that~$G$ is a finitely generated pro-$p$ group presented by relations~$l_{\bE}$ satisfying the Condition \eqref{condgradgroup}, then~$\E(G)=\E(\Gamma).$
\end{Theorem}

When~$\E(G)\simeq \E(\Gamma)$, we say that~$\E(G)$ is a Right Angled Artin Algebra (RAAAs). RAAAs play a fundamental role in geometric group theory (see for instance~\cite{bartholdi2020right}). In particular, since~$\E(\Gamma)$ is Koszul (see~\cite[Part~$4$]{bartholdi2020right}), we infer from Proposition \ref{compcoho}:
$$\ext_{\E(\Gamma)}^\bullet(\F_p;\F_p)\simeq \A(\Gamma),$$ 
where~$\A(\Gamma):=\E/\I^{!}(\Gamma)$, with~$\I^{!}(\Gamma)$ the two sided ideal of~$\E$ generated by the family 
  \begin{itemize}
\item~$X_iX_j$ when~$\{i,j\} \notin \mathbf{E}$,
\item~$X_u^2$ for~$u\in [\![1;d]\!]$,
\item~$X_uX_v+X_vX_u$ for~$u,v$ in~$[\![1;d]\!]$.
\end{itemize}
Observe that~$\dim_{\F_p}\A_n(\Gamma)=c_n(\Gamma)$, where~$c_n(\Gamma)$ is the number of~$n$-cliques of~$\Gamma$, i.e.\ complete subgraphs of~$\Gamma$ with~$n$ vertices. Since~$\E(\Gamma)$ is a Koszul algebra, we can apply Proposition~\ref{compcoho} and we infer that~$$H^\bullet(G)\simeq \A(\Gamma), \quad \text{and} \quad h^n(G):=\dim_{\F_p}H^n(G)=c_n(\Gamma).$$  

\subsection*{Outline}
We begin with Part~\ref{Part0}, where we give some backgrounds on Right Angled Artin Algebras (that we denote RAAA). Then we 
prove Theorem~\ref{gradgroup} in Part~\ref{Part1}. We finish by Part~\ref{lastpart}, where we first prove Proposition~\ref{compcoho}, then we compute the algebras~$\E(G)$ and~$H^\bullet(G)$ when~$G$ is free, mild quadratic and pro-$p$ RAAG. We conclude Part~\ref{lastpart} with the proof of Theorem~\ref{prescribed cohomological dimension}, which follows from Theorem~\ref{gradgroup} and Proposition~\ref{compcoho}.

\subsection*{Notation}
We introduce here some general notations:


$\bullet$ We recall that~$G$ is a finitely presented pro-$p$ groups with generators~$\{x_1;\dots;x_d\}$ and relations~$\{l_1;\dots;l_r\}$.

$\bullet$ If~$x,y$ are elements in~$G$ (or in~$F$), we denote by~$[x,y]:=x^{-1}y^{-1}xy$.

$\bullet$ We define~$H^n(G;\F_p)$ the~$n$-th (continuous) cohomology group of the trivial (continuous)~$G$-module~$\F_p$. The cohomological dimension of~$G$ is the integer~$n$ (which can be infinite) such that for every~$m>n$ we have~$H^m(G;\F_p)=0$.

$\bullet$ The Magnus isomorphism from~\cite[Chapitre II, Partie~$3$]{LAZ} gives us the following identification of~$\F_p$-algebras between~$E(F)$ and the noncommutative series over~$\F_p$ on~$\{X_1;\dots;X_d\}$ that we call~$E$: 
  \begin{equation}\label{Magnus iso}
\phi \colon E(F) \simeq E; \quad x_j \mapsto X_j+1.
\end{equation}
The algebra~$E$ is filtered by~$\{E_n\}_{n\in \NN}$, the~$n$-th power of the augmentation ideal, and we denote by~$F_n:=\{f\in F; \phi(f-1)\in E_n\}$ the Zassenhaus filtration of~$F$.

$\bullet$ Denote by~$I$ the closed two-sided ideal in~$E$ generated by~$w_i:=\phi(l_i-1)$, this is an algebra with a filtration given by~$\{I_{n}:=I\cap E_{n}\}_{n\in \NN}$. From the Magnus isomorphism \eqref{Magnus iso}, we identify the filtered algebra~$E(G)$ with the quotient algebra~$E/I$: this is a filtered algebra and we denote its filtration by~$\{E_{n}(G)\}_{n\in \NN}$. Let us define:
$$\E_n(G):=E_n(G)/E_{n+1}(G), \quad \text{and} \quad \E(G):=\bigoplus_n \E_n(G).$$ 

$\bullet$ We introduce the functor~$\grad$ (see for instance~\cite[Chapitre I]{LAZ}) from the category of compact~$\F_p$-vector spaces (or compact~$E(G)$-modules) to graded~$\F_p$-vector spaces (or graded~$\E(G)$-modules). This is an exact functor. For instance, if we denote by~$\E$ the noncommutative polynomials over~$\F_p$ on~$\{X_1;\dots;X_d\}$, and~$\E_n:=E_n/E_{n+1}$, we have
$$\grad(E):=\bigoplus_{n\in \NN}\E_{n}=\E.~$$

$\bullet$ Let us define~$\I:=\grad(I)=\bigoplus_n I_n/I_{n+1}$. Observe by~\cite[$(2.3.8.2),$ Chapitre I]{LAZ} that the functor~$\grad$ is exact, so from the Magnus isomorphism, we can identify~$\E(G)$ with the graded algebra~$\grad(E(G))\simeq \E/\I$, and we denote its gradation by~$\{\E_n(G)\}_{n\in \NN}$. We define the gocha series of~$G$ by:
  \begin{equation*}
  \begin{aligned}
&gocha(G,t) := & \sum_{n=0}^\infty c_n t^n, \quad &\text{where}& \quad c_n&:= \dim_{\F_p} \E_n(G)
\end{aligned}
\end{equation*}

$\bullet$ An~$\F_p$-basis on~$E$ and~$\E$ is given by monomials on the set of variables~$\mathbf{X}:=\{X_1;\dots;X_d\}$. The order~$X_1>X_2>\dots>X_d$ induces a lexicographic order on monomials on~$\mathbf{X}$, that we denote by~$>$. We say that a monomial~$X$ contains a monomial~$Y$ if there exist monomials~$M$ and~$N$ such that~$X=MYN$.
\\Recall that we write commutators of~$X_i$ and~$X_j$ (in~$E$ or~$\E$) as: 
$$[X_i;X_j]:=X_iX_j-X_jX_i \text{ for } \{i,j\}\in \bE.$$

$\bullet$ If~$z$ is an element in~$E$, we denote by~$\deg(z)$ the integer such that~$z\in E_{\deg(z)}\setminus E_{\deg(z)+1}$. Then we define~$\overline{z}$ the image of~$z$ in~$E_{\deg(z)}/E_{\deg(z)+1}$, this is a homogeneous polynomial, and we denote its degree by~$\deg(z)$. We call~$\widehat{z}$ the leading monomial of~$z$. For instance~$\widehat{[X_i;X_j]}=X_iX_j$.

$\bullet$ We say that~$G$ has a \textit{mild} presentation if the $\E/\I(\overline{w})$-module $\I(\overline{w})/\I(\overline{w})\I_{\E}$ is free on the family $\{\overline{w}_j\}_{j=1}^r$, where $\I_E$ (resp. $\I(\overline{w})$) is the augmentation ideal of $\E$, i.e. the two-sided ideal of $\E$ generated by the family $\{X_i\}_{i=1}^d$ (resp. the family $\{\overline{w_j}\}_{j=1}^r$).
\\The group~$G$ has a quadratic presentation if for every integer~$j$,~$\deg(w_j)=2$.

$\bullet$ We say that the algebra~$\E(G)$ is Koszul, if the trivial~$\E(G)$-module~$\F_p$ admits a (weighted) free linear resolution~$(\PP_{\bullet},\delta_{\bullet})$, i.e.~$\PP_i$ is a free-$\E(G)$-module generated by elements of degree~$i$ (see for instance~\cite[Chapter~$2$]{polishchuk2005quadratic}). 


\section{Preliminaries on Right Angled Artin Algebras (RAAA)}\label{Part0}
Recall that we denote by~$\Gamma:=(\bN, \bE)$ an undirected graph, where~$\bN:=[\![1;\dots d]\!]$. 
For every integer~$n$, we denote by~$c_n(\Gamma)$ the number of~$n$-cliques of~$\Gamma$. Let~$\I(\Gamma)$ (resp.\ $I(\Gamma)$) be the closed two sided ideal of~$\E$ (resp.\ $E$) generated by the family~$\{ [X_i;X_j]\}_{\{i,j\}\in \bE}$ and~$\E(\Gamma):=\E/\I(\Gamma)$ (resp.\ $E(\Gamma):=E/I(\Gamma)$).

We take the following orientation on~$\Gamma$, that we call \textit{standard}: if~$(i,j)\in \bE$ then~$i<j$. For more references on RAAAs, let us quote~\cite{bartholdi2020right}.

\subsection{Introductory results on graphs}
Let us begin with few results on graphs. I am thankful to Chris Hall for the following Lemma. We refer to~\cite[Chapters~$1$ and~$5$]{diestel2005graph} for a general introduction on graph theory.

\begin{lemm}\label{bipartite orientation1}
The undirected graph~$\Gamma$ is bipartite if and only if there exists an orientation on~$\Gamma$ such that the tail of an edge of~$\Gamma$ is not the head of another one. 
\end{lemm}

\begin{proof}
The undirected graph~$\Gamma$ is bipartite if and only if it is~$2$-colored (let us call these colors black and white). For instance, we can direct edges from white vertices to black vertices. 
\\Conversely, we can give a~$2$-coloring on a graph~$\Gamma$ where the tail of an edge is not the head of another one. We endow heads of edges with white color and tails of edges with black color.
\end{proof}

\begin{rema}[Bipartite graphs, orientation and combinatorially free families]\label{bipartite orientation2}
Up to relabelling vertices, we assume that bipartite graphs satisfy the property given in Lemma~\ref{bipartite orientation1} for the standard orientation. Equivalently, if~$\Gamma$ is bipartite, then the family~$\{X_iX_j\}_{(i,j)\in \bE}$ is combinatorially free in~$E$ (recall that we took the order on monomials induced by~$X_1>X_2>\dots>X_d$), i.e. for every different cuples~$(i_1,j_1)$ and~$(i_2,j_2)$ in~$\bE$, we have~$j_1\neq i_2$. 
\end{rema}

From now, we denote by~$\{i,j\}$ edges from the undirected graph~$\Gamma$ and~$(i,j)$ edges of the graph~$\Gamma$ endowed with its standard orientation. Of course, we only discuss cliques of the undirected graph~$\Gamma$. Let us give an example:

\begin{exem}
Consider~$\Gamma$ the graph with three vertices~$\{1;2;3\}$ and two edges~$\{\{1,2\}; \{1,3\}\}$. The graph~$\Gamma$ is bipartite and we have the following representation:

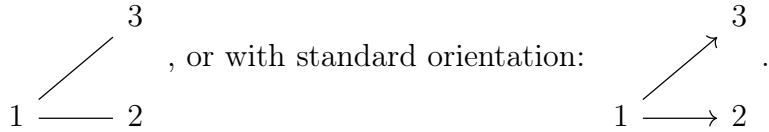
\centering{
\begin{tikzcd}
                                         & 3                   \\
1 \arrow[ru, no head] \arrow[r, no head] & 2
\end{tikzcd}
, or with standard orientation:
\begin{tikzcd}
                       & 3 \\
1 \arrow[ru] \arrow[r] & 2 
\end{tikzcd}.}
\end{exem}

\subsection{Gradation and RAAAs}
Let us begin with some introductory results on the functor~$\grad$ (for more references, see~\cite[Chapitre I]{LAZ}). We first show that the functor~$\grad$ sends homogeneous ideals (i.e. ideals generated by homogeneous polynomials) in~$E$ to homogenous ideals in~$\E$. 

Observe that~$E(\Gamma)$ is an augmented algebra, so filtered by powers of the augmentation ideal.

  \begin{lemm}[Gradation of~$E(\Gamma)$]\label{monoappro}
We have~$\grad(E(\Gamma))=\E(\Gamma)$.
\end{lemm}


  \begin{proof}
We just need to show that~$\grad(I(\Gamma))=\I(\Gamma)$. We always have~$\I(\Gamma) \hookrightarrow \grad(I(\Gamma))$. Let us show the reverse inclusion.

Take~$z\in I(\Gamma)$, and write~$z:=\sum_{ijul} a_{ijul}[X_i;X_j] b_{iju}$, where~$a,b\in E$. 
Let us express~$z$ as a (possibly infinite) sum of homogeneous polynomials:
$$a_{ijul}:=\sum_{g\in \NN}{}_ga_{ijul}, \quad \text{and} \quad b_{iju}:=\sum_{h\in \NN}{}_hb_{iju},$$
where~${}_ga_{ijul}$ and~${}_hb_{iju}$ are homogeneous polynomials of degree~$g$ and~$h$.
Therefore, we have the following (possibly infinite) sum of homogeneous polynomials:
$$z=\sum_{n\in \NN}\sum_{ijul}\sum_{g+h+2=n} ({}_{g}a_{ijul}) [X_i;X_j] ({}_{h}b_{iju}).$$
So, if~$\deg(z)=n$, we infer:
$$\overline{z}=\sum_{ijul}\sum_{g+h+2=n} ({}_{g}a_{ijul}) [X_i;X_j] ({}_{h}b_{iju})\in \I(\Gamma).$$
Therefore~$\grad(I(\Gamma))=\I(\Gamma)$ is a homogeneous ideal.
\end{proof}

  \begin{rema}
Lemma~\ref{monoappro} is still true if we take~$I$ a two-sided ideal in~$E$ generated by homogeneous elements~$w_u$ (which can be seen both in~$E$ and~$\E$). More precisely,~$\grad(I)$ will also be generated by~$w_u$ as a two-sided ideal in~$\E$.
\end{rema}

Recall, from the Condition~\eqref{condgradgroup}, that we defined~$w_{uv}:=\phi([x_u;x_v]-1)$ in~$E$. We compute here the homogeneous polynomials occurring in the expression of~$w_{uv}$.
  \begin{lemm}\label{comequa}
We have the following equality:
  \begin{equation*}
w_{uv}=\left(\sum_{n\in \NN}(-1)^n \sum_{k=0}^n P_{n,k}(X_u;X_v)\right)[X_u;X_v], \quad \text{where} \quad P_{n,k}(X_u;X_v)=X_u^kX_v^{n-k}.
\end{equation*}
\end{lemm}

  \begin{proof}
For every integer~$n$, we introduce the homogeneous polynomial of degree~$n$:~$P_n(X_u;X_v):=(-1)^n\sum_{k=0}^nP_{n,k}(X_u;X_v)\in E_n$. Let us observe that~$P_n$ satisfies the following equalities:
  \begin{equation} \tag{$\ast$} \label{(ast)}
  \begin{aligned}
P_n(X_u;X_v)&= X_u^n+P_{n-1}(X_u;X_v)X_v=X_v^n+P_{n-1}(X_u;X_v)X_u
\\&= X_u^n+X_v^n+P_{n-2}(X_u;X_v)X_uX_v.  
\end{aligned}
\end{equation}

Now, let us compute~$w_{uv}$. For this purpose, we introduce the series~$Z:=\sum_{n=1}^\infty (-1)^n P_n(X_u;X_v)$, and we infer:
  \begin{equation} \tag{$\ast\ast$} \label{(astast)} 
  \begin{aligned}
w_{uv}&= (1+X_u)^{-1}(1+X_v)^{-1}(1+X_u)(1+X_v)-1
 \\&= (1+Z)(1+X_u+X_v+X_uX_v)-1. 
\\&= X_u+X_v+X_uX_v+Z+Z(X_u+X_v)+ZX_uX_v. 
\end{aligned}
\end{equation}

Let us denote by~$w_{uv,n}$ the term (homogeneous polynomial) of degree~$n$ in~$w_{uv}$, i.e.\ $w_{uv}:=\sum_{n=1}^\infty w_{uv,n}$. Observe that:
$$w_{uv,1}=0, \quad \text{and} \quad w_{uv,2}=[X_u;X_v].$$
For~$n\geq 3$, we obtain from~\eqref{(astast)}:
$$w_{uv,n}=(-1)^n\left[ P_n(X_u;X_v)-P_{n-1}(X_u;X_v)(X_u+X_v)+P_{n-2}(X_u;X_v)X_uX_v \right].$$
We conclude by applying relations given in~\eqref{(ast)}.
\end{proof}

  \begin{prop}\label{commideals}
Denote by~$\Delta$ the ideal in~$E$ generated by~$\{w_{uv}:=\phi([x_u;x_v]-1); (u,v)\in \bE\}$. Then~$\Delta=I(\Gamma)$ and~$\grad(\Delta)=\I(\Gamma)$.
\end{prop}

  \begin{proof}
From Lemma \ref{comequa}, we can write $w_{uv}:=P_{uv}[X_u;X_v]$, where
$$P_{uv}:=\sum_{n\in \NN}(-1)^n \sum_{k=0}^n X_u^kX_v^{n-k} \text{ is invertible in } E.$$ Then $\Delta=I(\Gamma)$, and thus from Lemma \ref{monoappro}, we conclude that~$\grad(\Delta)=\grad(I(\Gamma)):=\I(\Gamma)$.  
\end{proof}


\section{Proof of Theorem~\ref{gradgroup}}\label{Part1}
The goal of this part is to compute~$\E(G)$, when~$G$ is presented by a family of relations~$l_{\bE}$ coming from a graph~$\Gamma$, endowed with standard orientation (see Remark~\ref{bipartite orientation1}), satisfying the Condition~$\eqref{condgradgroup}$. 
  \begin{theo}
Assume that~$G$ admits a presentation with relation~$l_{\bA \cup \bB}$ satisfying~\eqref{condgradgroup}. Then~$\E(G)= \E(\Gamma)$.
\end{theo}
We show that~$\I=\I(\Gamma)$. We split the proof into several steps. Using the proof of~\cite[Theorem~$3.7$]{FOR} we give Equalities~\eqref{decomposition} and~\eqref{homogenes} in subpart~\ref{firststep}. This allows us to express elements in~$I$ modulo~$E_{n+1}$ for every integer~$n$. The rest of the proof is done by contradiction.

In subpart~\ref{secondstep}, we infer Equalities~\eqref{mu equality} and \eqref{expression muA} from monomial analysis (Gröbner basis, see~\cite{Ufnarovskij1995CombinatorialAA}) and the fact that~$\widehat{w_{\bA}}:=\{X_iX_j\}_{(i,j)\in \bA}$ is combinatorially free. In subpart~\ref{thirdstep}, we show Equality \eqref{expression muB} from~$l_{uv}:=[x_u;x_v]$, Lemma~\ref{comequa} and Gröbner basis arguments. We finish the proof with subpart~\ref{laststep}, where we conclude that the contributions given by the homogeneous polynomials in the expressions of~$w_{ij}$ and~$w_{uv}$, for the computation of~$\I$, only come from the dominant terms. So we conclude~$\I=\I(\Gamma)$.

Recall that~$\widehat{w_{\bA}}:=\{X_iX_j\}_{(i,j)\in \bA}$ and~$\widehat{w_{\bB}}:=\{X_uX_v\}_{(u,v)\in \bB}$.
We introduce~$\I_{\bA}$ and~$\I_{\bB}$ the ideals in~$\E$ generated by~$\overline{w_{\bA}}:=\{[X_i;X_j]\}_{(i,j)\in \bA}$ and~$\overline{w_{\bB}}:=\{[X_u;X_v]\}_{(u,v)\in \bB}$. We denote by~$\widehat{\I(\Gamma)}$ (resp.~$\widehat{\I_{\bA}}$,~$\widehat{\I_{\bB}}$) the leading terms of a fixed Gröbner basis of~$\I(\Gamma)$ (resp.~$\I_{\bA}$,~$\I_{\bB}$), i.e.~$\widehat{\I(\Gamma)}$ is a set of generators of the ideal generated by the leading monomials of elements of~$\I(\Gamma)$. By Remark~\ref{bipartite orientation2}, we can take~$\widehat{\I_{\bA}}:=\widehat{w_{\bA}}$, furthermore since~$\I_{\bA}+\I_{\bB}=\I(\Gamma)$, we choose~$\widehat{\I(\Gamma)}$,~$\widehat{\I_{\bA}}$ and~$\widehat{\I_{\bB}}$ such that: 
\begin{equation}\tag{B0}\label{gb}
\widehat{w_{\bA}}\cup \widehat{w_{\bB}}\subset\widehat{\I_{\bA}}\cup \widehat{\I_{\bB}}\subset \widehat{\I(\Gamma)}, \quad \text{and} \quad \widehat{\I_{\bA}}:=\{X_iX_j\}_{(i,j)\in \bA}, \quad \widehat{w_{\bB}}\subset \widehat{\I_{\bB}}.
\end{equation}

\subsection{Decomposition}\label{firststep}
If~$A$ is a subset of~$E$, we recall that we have~$$\grad(A):=\bigoplus_n [(A\cap E_{n}+ E_{n+1})/E_{n+1}].$$
Furthermore,~$\grad(A)$ is a subset of~$\E$.

Observe that~$\I(\Gamma)\subset \I$. By~\cite[Theorems, Parts~$2.3$ and~$2.4$]{Ufnarovskij1995CombinatorialAA}, the ideal~$\I(\Gamma)$ admits a complementary subspace~$\mathcal{C}_{\Gamma}$ with a monomial basis given by monomials not containing any elements of~$\widehat{\I(\Gamma)}$. By Equation~\eqref{gb}, these monomials do not contain any elements of~$\widehat{w_{\bA}}\cup  \widehat{w_{\bB}}$. 

Furthermore, we denote the gradation on~$\mathcal{C}_{\Gamma}$ by~$\mathcal{C}_{\Gamma}:=\bigoplus_n\mathcal{C}_{\Gamma, n}$. Let us define by~$\mathcal{C}_n$ a complementary subspace of~$\I_n\cap \mathcal{C}_{\Gamma,n}$ in~$\mathcal{C}_{\Gamma,n}$, i.e.\ $\mathcal{C}_{\Gamma,n}=\mathcal{C}_n\bigoplus (\I_n\cap \mathcal{C}_{\Gamma,n})$.
\\Introduce~$\mathcal{C}:=\bigoplus_n \mathcal{C}_n$, this is a complementary subspace of~$\I$ in~$\E$, and every element~$c\in \mathcal{C}_n$ can be uniquely written as~$c=\sum_i c_i$, where~$c_i$ is a monomial of degree~$n$ in~$\mathcal{C}_{\Gamma,n}$. Denote by~$C:=\prod_n \mathcal{C}_n$ and~$C_{\Gamma}:=\prod_n \mathcal{C}_{\Gamma,n}$, these are filtered subsets of~$E$. By~\cite[Chapitre I,~$(2.3.7)$]{LAZ}, we have~$\grad(C)=\mathcal{C}$. 

In the beginning of the proof (first two pages) of~\cite[Theorem~$3.7$]{FOR}, Forré showed that~$C$ is a complementary subspace of~$I$ in~$E$, and for every integer~$n$, we have the following decomposition:
  \begin{equation}\tag{B1}\label{decomposition}
I=CWE+I^{n+1},
\end{equation}
where~$W$ is the~$\F_p$-vector space generated by~$w_{ij}:=\phi(l_{ij}-1)$, for~$(i,j)$ in~$\bA\cup \bB$.

Our goal is to show that~$\I=\I(\Gamma)$. Take~$f\in I$ of degree~$n$, we need to prove that~$\overline{f}$ (which describes a general element in~$\I$) is in~$\I(\Gamma)$. Using Equality \eqref{decomposition}, we can write:
  \begin{equation*}
  \begin{aligned}
f:=\sum_{(i,j)\in \bA}\sum_{k=1}^{n_{ij}}\sum_{l=1}^{n_{ijk}} s_{ijkl}+\sum_{(u,v)\in \bB}\sum_{o=1}^{n_{uv}}\sum_{q=1}^{n_{uvo}} s_{uvoq}+r_{n+1}, \quad \text{where}
\\s_{ijkl}=c_{ijkl}w_{ij} X_{ijk}, \quad s_{uvoq}=c_{uvoq}w_{uv}X_{uvo}, \quad \text{and} \quad r_{n+1} \in I^{n+1};
\end{aligned}
\end{equation*}
for~$c_\bullet$ in~$C$ and~$X_\bullet$ a monomial in~$E$. 

Therefore,
  \begin{equation}\tag{B2} \label{homogenes}
  \begin{aligned} 
f\equiv \sum_{\deg \leq n} s_{ijkl}+ \sum_{\deg \leq n} s_{uvoq} \pmod{E_{n+1}}.
\end{aligned}
\end{equation}

Then, without loss of generalities, we can write as a sum of monomials of degree less or equal than~$n$ in~$C_{\Gamma}$:~$c_{ijkl}:=\sum_{g=1}^{n_{ijkl}}c_{ijklg}, \quad \text{and} \quad
c_{uvoq}:=\sum_{h=1}^{n_{uvoq}}c_{uvoqh}.$
\\From now, we simplify the notations on indices by denoting~$c_{ijkl}$ and~$c_{uvoq}$ as monomials in~$C_{\Gamma}$. Recall by Lemma~\ref{comequa} that we have the following sum of homogeneous polynomials:
$$w_{uv}:=\sum_{r=2}^\infty\sum_{z=0}^r w_{uv rz}, \quad \text{with } w_{uv rz}:=P_{r-2,z}(X_u;X_v)[X_u;X_v],$$ 
where~$w_{uv rz}$ is of degree~$r$.

A natural candidate for~$\overline{f}$ would be~$\sum_{\deg \leq n} c_{ijkl}[X_i;X_j]X_{ijk}+ \sum_{\deg \leq n} c_{uvoq}[X_u;X_v]X_{uvo}$. However, the terms in the previous sums can be of degree strictly less than~$n$. We then work on degree arguments. Especially, we shall study particular leading monomials and so we shall choose special indices, that we will denote by bold letters.

\subsection{Monomial analysis}\label{secondstep}
Similarly to the proof of~\cite[Theorem~$3.7$]{FOR}, we introduce 
$m_{\bA}:=\inf_{ijkl, (i,j)\in \bA}(\deg(s_{ijkl})).$ The goal of the rest of the proof is to show that~$m_{\bA}=n$, then we conclude that this equality allows us to show that~$\overline{f}$ is in~$\I(\Gamma)$.
We argue by contradiction to show that~$m_{\bA}=n$. Assume that~$m_{\bA}<n$, then from Equality \eqref{homogenes}, we infer:
$$\sum_{\deg=m_{\bA}} c_{ijkl}[X_i;X_j] X_{ijk}+ \sum_{\deg=m_{\bA}} c_{uvoq}w_{uv rz} X_{uvo}=0.$$
Furthermore, by definition of~$m_{\bA}$, we can assume for every~$(i,j)$ in~$\bA$ and~$k$ that~$\sum_l c_{ijkl}\neq 0$. Define~$\mu_{\bA}$ and~$\mu_{\bB}$ by
$$\mu_{\bA}:=\sum_{\deg=m_{\bA}} c_{ijkl}[X_i;X_j] X_{ijk}, \quad \text{and} \quad \mu_{\bB}:= \sum_{\deg=m_{\bA}} c_{uvoq}w_{uv rz} X_{uvo}.$$

Before studying the polynomials~$\mu_{\bA}$ and~$\mu_{\bB}$, we bring back some results on strongly and combinatorially free families from~\cite{FOR}. Recall that~$\I_{\bA}$ is the ideal of~$\E$ generated by~$\overline{w_{\bA}}:=\{[X_i;X_j]\}_{(i,j)\in \bA}$, and denote by~$\mathcal{J}_{\bA}$ the ideal of~$\E$ generated by~$\widehat{w_{\bA}}$. Since~$\widehat{w_{\bA}}$ is combinatorially free, we infer that~$\widehat{w_{\bA}}$ is a Gröbner basis of~$\I_{\bA}$, and by~\cite[Theorem~$2.6$]{FOR} the family~$\overline{w_{\bA}}$ is strongly free, i.e. if we denote by~$\E_{\geq 1}$ the augmentation ideal of~$\E$, the~$\E/\I_{\bA}$-module~$\I_{\bA}/\I_{\bA}\E_{\geq 1}$ is free over~$\overline{w_{\bA}}$.  Moreover, by~\cite[Theorem~$2.3$]{FOR}, the family~$\widehat{w_{\bA}}$ is a basis of the free~$\E/\mathcal{J}_{\bA}$-module~$\mathcal{J}_{\bA}/\mathcal{J}_{\bA}\E_{\geq 1}$. 

Let us define~$\mathcal{C}_{\bA}$ the subspace of~$\E$ generated by all monomials not containing any elements of~$\widehat{w_{\bA}}$. By~\cite[Theorems Parts~$2.3$ and~$2.4$]{Ufnarovskij1995CombinatorialAA}, we notice that the~$\F_p$-vector space~$\mathcal{C}_{\bA}$ is both a complementary subspace of~$\I_{\bA}$ and~$\mathcal{J}_{\bA}$. From that fact, we can apply the strategy used in~\cite[Theorem~$3.7$, beginning of the page~$181$]{FOR}.

If~$\mu_{\bB}=0$, then~$\mu_{\bA}=0$. The proof of~\cite[Theorem~$3.7$, beginning page~$181$]{FOR} shows that this case is impossible, since~$\{[X_i;X_j]\}_{(i,j)\in \bA}$ is strongly free and~$c_{ijkl}$ does not contain any monomials in~$\widehat{w_{\bA}}$ so is in~$\mathcal{C}_{\bA}$.  Consequently,~$\mu_{\bB}$ and~$\mu_{\bA}$ are both different from zero. This implies that 
  \begin{equation}\tag{B3} \label{mu equality}
  \begin{aligned}
\widehat{\mu_{\bA}}=\widehat{\mu_{\bB}}\neq 0
\end{aligned}
\end{equation}

We study now the structure of the monomials~$\widehat{\mu_{\bA}}$ and~$\widehat{\mu_{\bB}}$. 
\\ From Remark~\ref{bipartite orientation1}, the family~$\widehat{\I_{\bA}}$ is combinatorially free, then it is strongly free (see~\cite[Theorem~$2.3$]{FOR}). Using a similar argument as~\cite[Beginning of page~$181$]{FOR}, 
we infer that~$\widehat{\mu_{\bA}}=c_{\bi\bj\bk\bl}X_{\bi}X_{\bj}X_{\bi\bj\bk}$ for some fixed coefficients~$(\bi,\bj)\in \bA$ and~$\bk, \bl$. 
Indeed if the previous equality does not hold, there exists a relation of the form:
$$\sum_{ijkl} c_{ijkl}X_iX_j X_{ijk}=0.$$
Since $\{X_1;\dots;X_d\}$ is a $\E$-linearly independant family, we can assume that at least one monomial~$X_{\bi\bj\bk}$ has valuation zero (so is in~$\F_p$), then we obtain a relation:
$$\sum c_{ijkl}X_iX_jX_{ijk} \equiv 0 \pmod{\mathcal{J}_{\bA}\E_{\geq 1}}.$$
Since~$\widehat{w}_{\bA}$ is strongly free, we infer that~$\sum_{l}c_{\bi\bj\bk l}$ is in $\mathcal{C}_{\bA}\cap \J_{\bA}=\{0\}$. This is a contradiction.
Consequently, we can write:
  \begin{equation}\tag{B4} \label{expression muA}
\begin{aligned}
\widehat{\mu_{\bA}}:=M_{\bA}X_{\bi}X_{\bj} X_{\bA}
\end{aligned}
\end{equation}
where~$M_{\bA}:=c_{\bi\bj\bk\bl}$ and~$X_{\bA}:=X_{\bi\bj\bk}$. Observe that~$M_{\bA}$ is a monomial in~$\mathcal{C}_{\Gamma}$, so from~\eqref{gb}, the monomial~$M_{\bA}$ does not contain any monomials in~$\widehat{w_{\bA}}\cup \widehat{w_{\bB}}$.

Recall, from hypothesis, that~$m_{\bA}<n:=\deg(f)$. Let us show that~$\widehat{\mu_{\bB}}$ has the following form:
  \begin{equation}\tag{B5} \label{expression muB}
  \begin{aligned}
\widehat{\mu_{\bB}}:=M_{\bB}X_{\bu}X_{\bv} X_{\bB}
\end{aligned}
\end{equation}
for some fixed~$(\bu,\bv)$ in~$\bB$, some monomial~$X_{\bB}$ and some monomial~$M_{\bB}$ not containing submonomials in~$\widehat{w_{\bA}}\cup \widehat{w_{\bB}}$. From Lemma~\ref{comequa},~$\widehat{\mu_{\bB}}$ has one of the following forms, for some fixed index~$(\bu,\bv)\in \bB$:
$$(a)\quad \widehat{\mu_{\bB}}=c_{\bu\bv\bo\bq}P_{\br-2,\bz}(X_{\bu},X_{\bv})X_{\bu}X_{\bv} X_{\bu\bv\bo}, \quad \text{or} \quad (b) \quad \widehat{\mu_{\bB}}=c_{\bu\bv\bo\bq}P_{\br-2,\bz}(X_{\bu},X_{\bv})X_{\bv}X_{\bu} X_{\bu\bv\bo}.$$

The monomial~$P_{\br-2,\bz}(X_{\bu},X_{\bv})$ contains a monomial of the form~$X_uX_v$ with~$(u,v)\in \bB$ if and only if~$0<\bz<\br-2$. Observe that~$\widehat{\mu_{\bB}}$ has one of the following from:
\begin{itemize}
\item[$(i)$]
the case~$(a)$,
\item[$(ii)$]
the case~$(b)$ with~$0<\bz\leq \br-2$,
\item[$(iii)$]
the case~$(b)$ with~$0=\bz$, but~$c_{\bu\bv\bo\bq}:=XX_u$ where~$X$ is a monomial and~$(u,\bv)\in \bB$,
\item[$(iv)$]
the case~$(b)$ with~$0=\bz$, and~$c_{\bu\bv\bo\bq}$ does not finish by~$X_u$ such that~$(u,\bv)\in \bB$.
\end{itemize}
For the case~$(i)-(iii)$, we always infer a monomial~$M_{\bB}$ not containg a submonomial in~$\widehat{w_{\bA}}\cup \widehat{w_{\bB}}$ such that~$\widehat{\mu}_{\bB}=M_{\bB}X_uX_vX_{\bB}$, so a positive solution to Equation~\eqref{expression muB}. The case $(i)$ is obvious, and for cases $(ii)$ and $(iii)$, we observe that the monomial~$c_{\bu \bv \bo \bq}P_{r-2,z}(X_{\bu}, X_{\bv})X_{\bv}$ contains a monomial $X_{\bu}X_{\bv}$, which allows us to conclude. In the next subpart, we show that the case~$(iv)$ is impossible, which allows us to infer~\eqref{expression muB}.

\subsection{Structure of~$\widehat{\mu_{\bB}}$.}\label{thirdstep}
To conclude, under the hypothesis~$m_{\bA}<n$, we show that the case~$(iv)$ is impossible. By contradiction, we assume that
  \begin{equation*}
  \begin{aligned}
\widehat{\mu_{\bB}}=c_{\bu\bv\bo\bq}X_{\bv}^{\br-2} X_{\bv}X_{\bu} X_{\bu\bv\bo}, \quad \text{for some integer~$\br$, and} 
\\ \text{$c:=c_{\bu\bv\bo\bq}X_{\bv}^{\br-2}$ does not contain a monomial in~$\widehat{w_{\bA}}\cup \widehat{w_{\bB}}$.}
\end{aligned}
\end{equation*}
By Equalities~\eqref{expression muA} and~\eqref{mu equality}, we infer:
$$\widehat{\mu_{\bB}}=cX_{\bv}X_{\bu}X_{\bu\bv\bo}=c_{\bi\bj\bk\bl}X_{\bi}X_{\bj}X_{\bi\bj\bk}.$$
Since~$cX_{\bv}X_{\bu}$ does not contain a monomial in~$\widehat{w_{\bA}}$, we infer that there exists a monomial~$X'_{\bu\bv\bo}$ include in~$c_{\bi\bj\bk\bl}$
such that
$$X_{\bu\bv\bo}=X'_{\bu\bv\bo}X_{\bi}X_{\bj}X_{\bi\bj\bk}.$$
Consider the following restricted sum~$\mu_{\bB}'$ of~$\mu_{\bB}$ where every polynomial of degree~$m_{\bA}$ finishes by~$X_{\bi}X_{\bj}X_{\bi\bj\bk}$, and~$\bi,\bj,\bk$ is fixed from~\eqref{expression muA} (here~$X_{\bi}X_{\bj}X_{\bi\bj\bk}$ is the end of~$\widehat{\mu_{\bA}}$): 
$$\mu_{\bB}'=\sum_{\deg=m_\bA}\sum_{(u,v)\in \bB} \sum_{o,q,r,z}\sum_{X_{uvo}=X_{uvo}'X_{\bi}X_{\bj}X_{\bi\bj\bk}} c_{uvoq}w_{uvrz}X_{uvo},$$ 
This sum is not empty, every term in that sum finishes by~$X_{\bi}X_{\bj}X_{\bi\bj\bk}$, and that sum is in~$\I_{\bB}$: the two-sided ideal of~$\E$ generated by~$\overline{w_{\bB}}$. Observe that~$\widehat{\mu_{\bB}}=\widehat{\mu_{\bB}'}$.

Define~$\mu_{\bB}''$ by~$\mu_{\bB}':=\mu_{\bB}''(X_{\bi}X_{\bj}X_{\bi\bj\bk})$. Notice that~$\mu_{\bB}''$ is in~$\I_{\bB}$. Therefore~$\mu_{\bB}''$ contains a monomial in~$\widehat{\I_{\bB}}$. Furthermore, by definition 
$$\widehat{\mu_{\bB}}=\widehat{\mu_{\bB}'}=\widehat{\mu_{\bB}''}X_{\bi}X_{\bj}X_{\bi\bj\bk}=c_{\bi\bj\bk\bl}X_{\bi}X_{\bj}X_{\bi\bj\bk},$$ 
consequently~$\widehat{\mu_{\bB}''}=c_{\bi\bj\bk\bl}$ is in~$\CC_{\Gamma}$ and therefore by \eqref{gb} does not contain monomials in~$\widehat{\I_{\bB}}$. This is impossible. We studied all cases, so we conclude that~$\widehat{\mu_{\bB}}$ satisfies Equality~\eqref{expression muB}.

\subsection{Conclusion}\label{laststep}
Let us first show that~$m_{\bA}=n$. If~$m_{\bA}<n$, then from Equalities \eqref{mu equality}, \eqref{expression muA} and \eqref{expression muB}, we have:
$$M_{\bA}X_{\bi}X_{\bj}X_{\bA}=M_{\bB}X_{\bu}X_{\bv}X_{\bB}.$$
Therefore,~$M_{\bA}=M_{\bB}$. This is impossible since~$X_{\bi}\neq X_{\bu}$.
We conclude that~$m_{\bA}=n$.

Let us now finish our proof, by showing that~$\overline{f}$ is in~$\I(\Gamma)$. Using Equality \eqref{homogenes}, we have modulo~$E_{n+1}$:
$$f= \sum_{\deg=n,ijkl, (i,j)\in \bA} s_{ijkl}+\sum_{\deg\leq n, uvoq, (u,v)\in \bB}s_{uvoq}.$$
Since~$f$ and~$\sum_{\deg=n} s_{ijkl}$ are both of degree~$n$, then~$\sum_{\deg\leq n}s_{uvoq}$ is at least of degree~$n$, and by Lemma~\ref{monoappro} we have~$\overline{\sum_{\deg\leq n}s_{uvoq}}\in  \I(\Gamma)$. Consequently modulo~$E_{n+1}$, we infer:
  \begin{multline*}
f\equiv  \overline{\sum_{\deg=n} s_{ijkl}+\sum_{\deg\leq n}s_{uvoq}}= \overline{\sum_{\deg=n} s_{ijkl}}+\overline{\sum_{\deg\leq n}s_{uvoq}}
\\ \equiv \sum_{\deg=n} c_{ijkl}[X_i;X_j]X_{ijkl}+\overline{\sum_{\deg\leq n}s_{uvoq}}.
\end{multline*}
Thus~$\overline{f}\in \I_n(\Gamma)$, so~$\I(\Gamma)=\I$.

  \begin{rema}\label{gradgroup2}
In the proof of Theorem~\ref{gradgroup}, we constructed a filtered~$\F_p$-vector space~$C_\Gamma$, and we showed that if~$\I=\I(\Gamma)$, then~$E(G)$ is isomorphic to~$C_{\Gamma}$ as a filtered~$\F_p$-vector space. In fact, we can define an algebra structure on~$C_{\Gamma}$ using the natural surjection~$\phi\colon E\to E(G)$ induced by the Magnus isomorphism and show that~$C_\Gamma$ is indeed isomorphic (as a filtered algebra) to~$E(G)$. 
\end{rema}

  \begin{rema}[Gocha series and filtrations for groups satisfying the Condition~\eqref{condgradgroup}]
We assume that~$G$ admits a presentation which satisfies the Condition \eqref{condgradgroup}. 
The gocha series of~$G$ is given by:
$$gocha(G,t)=\frac{1}{\sum_{k=0}^n (-1)^k c_k(\Gamma)t^k}, \quad \text{and} \quad h^n(G)=c_n(\Gamma), \text{ for every integer~$n$}.$$


Let us denote by~$a_n:=\dim_{\F_p}G_{n}/G_{n+1}$.
Then using~\cite[Theorem~$2.9$]{Minac}, we can explicitly compute~$a_n$ for every integer~$n$. See also~\cite{split} for an equivariant study.
\end{rema}
 
\subsection{Example}\label{basic example}
Let us give an example:

We define~$\Gamma$ a graph with~$6$ vertices and five edges given by~$\bE:=\bA\sqcup \bB$, where~$\bA:=\{(1,2); (1,3)\}$ and~$\bB:=\{(4,5);(4,6);(5,6)\}$. A representation of~$\Gamma$ is given by:

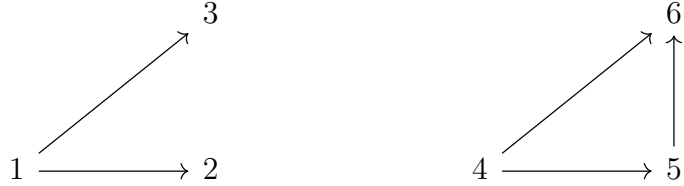
\centering{\begin{tikzcd}
                          &  & 3 &  &  &                           &  & 6            \\
                          &  &   &  &  &                           &  &              \\
1 \arrow[rr] \arrow[rruu] &  & 2 &  &  & 4 \arrow[rr] \arrow[rruu] &  & 5 \arrow[uu]
\end{tikzcd}}

\justifying
Take~$G$ a pro-$p$ group defined by six generators and five relations of the form~$l_{\bA\cup\bB}$ given by: 
  \begin{equation*}
  \begin{aligned}
l_{12}\equiv 1+[X_1;X_2] \pmod{E_3}, \quad \text{and} \quad l_{13}\equiv 1+[X_1;X_3] \pmod{E_3},
\\ l_{45}:=[x_4;x_5], \quad l_{46}:=[x_4;x_6], \quad \text{and} \quad l_{56}:=[x_5;x_6].
\end{aligned}
\end{equation*}
Observe that~$\Gamma$ and the relations~$l_{\bE}$ satisfy the~Condition~\eqref{condgradgroup}.
Therefore, by Theorem~\ref{gradgroup}, the algebra~$\E(G)$ is given by~$\E(\Gamma):=\E/\I(\Gamma)$, where 
$$\I(\Gamma):=\langle [X_1;X_2], [X_1;X_3], [X_4;X_5], [X_4;X_6], [X_5;X_6] \rangle.$$
Furthermore, thanks to Proposition~\ref{compcoho}, that we prove in Part~\ref{lastpart}, we have:
$$h^1(G)=c_1(\Gamma)=6, \quad h^2(G)=c_2(\Gamma)=5, \quad h^3(G)=c_3(\Gamma)=1, \quad \text{else } h^n(G)=c_n(\Gamma)=0.$$
Consequently~$G$ has cohomological dimension~$3$.

\section{Applications to pro-$p$ groups with quadratic presentation}\label{lastpart}
In this part, we begin to prove Proposition~\ref{compcoho}, then we illustrate it with some examples. We say that~$G$ has a \textit{quadratic presentation} if it is presented by a family of quadratic relations~$l:=\{l_i\}$ (i.e.\ $l_i$ is in~$F_2\setminus F_3$).

\subsection{Proof of Proposition~\ref{compcoho}}
I am thankful to Thomas Weigel for the following argument. We also refer to~\cite{weigelspectralsequence} for further details.

Let us denote by~$\Delta_{\bullet}(G)$ the graded algebra indexed by negative integers:~$\Delta_\bullet(G):=\bigoplus_i \Delta_i(G)$ where~$\Delta_i(G):=\E_{-i}(G)$.
Following notations from Theorem~\cite[Theorem~$5.1.12.(2)$]{symonds2000cohomology} and its proof, if the algebra~$\E(G)$ is Koszul then~$\ext_{\Delta_{\bullet}(G)}^{\bullet,\bullet}$ is the quadratic dual of~$\E(G)$ generated by~$X_1,\dots, X_d$ where every~$X_i$ is endowed with bidegree~$(-1,2)$. In particular,~$\ext_{\Delta_{\bullet}(G)}^{s,t}\neq 0$ only if~$t=-2s$.

From Theorem~\cite[Theorem~$5.1.12.(2)$]{symonds2000cohomology}, we infer a spectral sequence~$(E_r^{\bullet,\bullet};d_r)$ and a filtration~$F^\bullet$ on~$H^\bullet(G)$ such that:
\begin{itemize}
\item[$\bullet$]
$E_1^{\bullet,\bullet}=\ext_{\Delta_{\bullet}(G)}^{\bullet,\bullet}(\F_p,\F_p)$,
\item[$\bullet$]~$E_{\infty}^{s,t}=F^sH^{s+t}(G)/F^{s+1}H^{s+t}(G)$.
\end{itemize}
In particular, we have~$d_1=0$, so we infer an isomorphism of graded algebras~$E_1^{\bullet,\bullet}\simeq E_{\infty}^{\bullet,\bullet}$.
The filtration~$F^{\bullet}$ on~$H^\bullet(G)$ is decreasing and from the convergence of the spectral sequence, we obtain:
$$\dots \supset F^{-(n+1)}H^n(G)=H^n(G) \supset F^{-n} H^n(G)=H^n(G)\supset F^{-(n-1)}H^n(G)=0 \dots$$
Consequently, we infer the following isomorphism of graded algebras:
$$H^\bullet(G;\F_p)\simeq \ext_{\E(G)}^\bullet(\F_p;\F_p).$$

\begin{rema}
We propose an alternative proof, using Serre's Lemma~\cite[Partie~$5$, Lemme~$2.1$]{LAZ}, of the fact that we have an isomorphism of graded vetor spaces between~$H^\bullet(G;\F_p)$ and~$\ext_{\E(G)}^\bullet(\F_p;\F_p)$. 

Let~$\PP:=(\PP_i, \delta_i)$ be a Koszul resolution of~$\F_p$, then there exists a~$E(G)$-free resolution~$P:=(P_i,d_i)$ of~$\F_p$ such that~$\grad(P):=(\grad(P_i), \grad(d_i))=\PP$, i.e.\ for every~$i$,~$\grad(P_i)=\PP_i$ and~$\grad(d_i)=\delta_i$. Moreover, there exits a family~$p_{i,j}$ in~$P_i$ such that 
$$P_i:=\prod_j p_{i,j}E(G) \quad \text{and} \quad \PP_i:=\prod_j \overline{p_{i,j}}\E(G).$$
Since~$P_i$ (resp.\ $\PP_i$) is a free compact~$E(G)$-module (resp.\ graded~$\E(G)$-module), we infer two isomorphisms of discrete~$\F_p$-vector spaces:
$$Hom_{E(G)}(P_i;\F_p)\simeq \bigoplus_j p_{ij}^\ast\F_p, \quad \text{and} \quad Hom_{\E(G)}(\PP_i;\F_p)\simeq \bigoplus_j \overline{p_{ij}}^\ast\F_p,$$
where~$p_{ij}^\ast$ (resp.\ $\overline{p_{ij}}^\ast$) is the function which maps~$\sum_l p_{il}e_l\in P_i$ with~$e_l\in E(G)$ (resp.\ $\sum_l\overline{p_{il}}f_l\in \PP_i$, with~$f_l\in \E(G)$) to~$\epsilon(e_j)$ (resp.\ $\epsilon(f_{j})$), for~$\epsilon$ the augmentation map of~$E(G)$ (or~$\E(G)$).

Define by~$gr\colon Hom_{E(G)}(P_i;\F_p)\to Hom_{\E(G)}(\PP_i;\F_p)$ the morphism of~$\F_p$-vector spaces which maps~$p_{ij}^\ast$ to~$\overline{p_{ij}}^\ast$. We infer the following diagram of discrete~$\F_p$-vector spaces:

\centering{
  \begin{tikzcd}
Hom_{E(G)}(P_{i+1};\F_p) \arrow[d,"gr "] && Hom_{E(G)}(P_i;\F_p) \arrow[ll,"d_{i+1}^\ast "] \arrow[d,"gr "] && Hom_{E(G)}(P_{i-1};\F_p) \arrow[d,"gr"] \arrow[ll, "d_{i}^\ast "] \\
Hom_{\E(G)}(\PP_{i+1};\F_p)        &  & Hom_{\E(G)}(\PP_i;\F_p) \arrow[ll,"\delta_{i+1}^\ast "]         &  & Hom_{\E(G)}(\PP_{i-1};\F_p) \arrow[ll,"\delta_{i}^\ast"]           
\end{tikzcd}
}

\justifying Observe that the previous diagram is in general not commutative. Since the resolution~$\PP$ is Koszul, we show that the previous diagram is indeed commutative. More precisely, we show that for every~$i$, the map~$d_i^\ast$ is zero.

Since~$d_i$ is a filtered morphism, we can write~$d_i(p_{i,l}):=\sum_{m} p_{i-1,m}\sum_{k=1}^d \alpha_{k,m}X_k +c_{i,l}$ with~$c_{i,l}$ an element of degree strictly larger than~$i$ in~$P_{i-1}$, and~$c_{i,l}:=\sum_m p_{i-1,m}u_m$. In particular,~$\epsilon(u_m)=0$. Consequently, we have:
\begin{equation*}
\begin{aligned}
d_i^\ast(p_{i-1,j}^\ast)(p_{i,l})&=p_{i-1,j}^\ast\circ d_i(p_{i,l})
							   \\&=p_{i-1,j}^\ast\left(\sum_{m} p_{i-1,m}\sum_{k=1}^d \alpha_{k,m}X_k +c_{i,l}\right)
							   \\&=p_{i-1,j}^\ast\left(\sum_{m} p_{i-1,m}(\sum_{k=1}^d \alpha_{k,m}X_k +u_m)\right)
							   \\&=\epsilon(\alpha_{k,j}X_k +u_j)
							   \\&=0,
\end{aligned}
\end{equation*}
therefore we have~$d_i^{\ast}=0$.
\end{rema}

\subsection{Free pro-$p$ groups}
Assume that~$G$ is a free pro-$p$ group, then by the Magnus isomorphism, we infer~$\E(G)\simeq \E$. Using Proposition~\ref{compcoho}, we obtain the well known result:
$$H^\bullet(G)\simeq \ext^\bullet_{\E}(\F_p;\F_p), \text{ so} \quad H^n(G)=0 , \text{ for } n \geq 2.$$
Thus $H^{\bullet}(G)$ is generated in degree one.

\subsection{Mild quadratic pro-$p$ group}
In this subsubsection, we slightly improve~\cite[Theorem~$1.3$]{minavc2022mild}.

From~\cite[Theorem~$3.7$]{FOR}, if~$G$ has a mild quadratic presentation, then~$\E(G)$ is a quadratic algebra. In fact, in the proof of~\cite[Theorem~$1.3$]{minavc2022mild}, Min{\'a}{\v{c}}-Pasini-Quadrelli-T{\^a}n showed that the algebra~$\E(G)$ is Koszul. Denote its quadratic dual by~$\A(G)$ (see for instance~\cite[Chapter~$1$, Part~$2$]{polishchuk2005quadratic} for more details).

  \begin{coro}
Assume that~$G$ has a mild quadratic presentation. Then~$H^\bullet(G)$ and~$\E(G)$ are both quadratic algebras. Furthermore, we have:
$$H^\bullet(G)\simeq \A(G).$$
\end{coro}

  \begin{proof}
Since~$\E(G)$ is Koszul, we can apply Proposition~\ref{compcoho}. We infer~$$H^\bullet(G)\simeq \ext_{\E(G)}^\bullet(\F_p;\F_p).$$ Furthermore~$\ext_{\E(G)}^\bullet(\F_p;\F_p)\simeq \A(G)$.
Consequently:
$$H^\bullet(G)\simeq \ext^\bullet_{\E(G)}(\F_p;\F_p)\simeq \A(G).$$
\end{proof}

\subsection{Pro-$p$ Right Angled Artin Groups}
Let us recall that $\Gamma$ is an undirected graph with set of vertices~$\{1;\dots;d\}$, and set of unidrected edges denoted by~$\bE:=\{(i,j)\}$. We say that~$G_\Gamma$ is a Right Angled Artin Group (RAAG) if~$G_\Gamma$ admits a presentation~$\mathcal{F}/S_\Gamma$ where~$\mathcal{F}$ is the abstract free group on~$\{x_1;\dots;x_d\}$ and~$S_\Gamma$ is a normal subgroup of~$\mathcal{F}$ generated by~$[x_i;x_j]$ for~$(i,j)\in \bE$. 

We say that~$G(\Gamma)$ is pro-$p$ RAAG if~$G(\Gamma)$ is the pro-$p$ completion of~$G_\Gamma$. The pro-$p$ group~$G(\Gamma)$ admits a presentation~$F/R_\Gamma$ where~$F$ is a free pro-$p$ group over~$\{x_1;\dots;x_d\}$ and~$R_\Gamma$ is a closed normal subgroup of~$F$ generated by ~$[x_i;x_j]$ for~$(i,j)\in \bE$. 

The algebra~$H^\bullet(G(\Gamma))$ is already known. Lorensen~\cite[Theorem~$2.7$]{lorensen2010groups} showed that~$$H^\bullet(G(\Gamma))\simeq H^\bullet(G_\Gamma).$$
It is also well-known, see~\cite{bartholdi2020right}, that~$H^\bullet(G_\Gamma)\simeq \A(\Gamma)$. Consequently

  \begin{theo}\label{propcoho}
Let~$G(\Gamma)$ be pro-$p$ RAAG, then we have the following isomorphism:
$$H^\bullet(G(\Gamma))\simeq \A(\Gamma).$$
\end{theo}
We propose another proof of Theorem~\ref{propcoho}.

  \begin{prop}\label{propcoho2}
Let~$G$ be a pro-$p$ RAAG with underlying graph~$\Gamma$, then we have~$E(G)=E(\Gamma)$. Therefore, we infer:
$$\E(G)\simeq \E(\Gamma), \quad \text{and} \quad H^\bullet(G(\Gamma)) \simeq \A(\Gamma).$$
\end{prop}

  \begin{proof}
Here, we just need to observe, following notations of Proposition~\ref{commideals}, that~$I=\Delta$. Then we infer, using Proposition~\ref{commideals}, that~$E(G)=E(\Gamma)$. From Lemma~\ref{monoappro} and Proposition~\ref{commideals}, we conclude that~$\E(G)=\E(\Gamma)$.

Consequently,~$\E(G)$ is quadratic and Koszul. We finish the proof using Proposition~\ref{compcoho}.
\end{proof}

  \begin{rema}
Observe that the~$\F_p$-vector space~$C_{\Gamma}$ constructed in Subpart~\ref{gradgroup2} does depend only on~$\Gamma$. In particular, using Remark~\ref{gradgroup2} and Proposition~\ref{propcoho2}, we conclude that the filtered vector space~$C_{\Gamma}$ is isomorphic to the filtered vector space~$E(\Gamma)$.
\end{rema}

\subsection{Prescribed and restricted ramification}
We finish this paper by showing a more precise version (and a proof) of Theorem~\ref{prescribed cohomological dimension}:

  \begin{theo}[Galois extensions with prescribed ramification and cohomology]\label{prescribed cohomological dimension2}
Fix~$\Gamma$ and~$l_{\bE}$ satisfying the Condition \eqref{condgradgroup}. Then, there exist a totally imaginary field~$K$ and a set~$T$ of primes in~$K$ such that~$G_K^T$, the Galois group of the maximal pro-$p$ extension of~$K$ unramified outside~$p$ and which totally splits in~$T$, is presented by set of generators~$\{x_1;\dots;x_d\}$ and set of relations~$l_{\bA}$ satisfying the Condition~\eqref{condgradgroup}.
 
Furthermore, there exist a quotient~$G$ of~$G_K^T$ satisfying~$\E(G)\simeq \E(\Gamma)$.
In particular,~$H^\bullet(G)\simeq \A(\Gamma)$. 
\end{theo}

  \begin{proof}
Take~$k:=\Q(\sqrt{-p})$ and define~$k_p$ the maximal~$p$-extension of~$k$ unramified outside places above~$p$ in~$k$. 
From~\cite[Proof of Corollary~$4.6$]{maire2023galois} we observe that~$p$ is coprime to the class number of~$k$. Consequently, from~\cite[Theorems~$11.5$ and~$11.8$]{Koch} we infer that~$G_k:=\Gal(k_p/k)$ is a free pro-$p$ group with~$2$ generators. 

Let~$F$ be an open subgroup of~$G_k$ with index~$|G_k:F|$ larger than~$d$. Then using the Schreier formula (see~\cite[Theorem~$3.3.16$]{NSW}), we infer that the group~$F$ is pro-$p$ free with~$d':=1+|G_k:F|$  generators. Let~$K$ be the fixed subfield of~$k_p$ by~$F$. Observe that~$K_p$, the maximal~$p$-extension of~$K$ unramified outside places in~$K$ above~$p$, is equal to~$k_p$, and so~$F:=\Gal(k_p/K)=G_K$. We define~$V':=[\![d+1;d']\!]$.

By the Chebotarev Density Theorem (see for instance~\cite[Part~$2$]{split}), there exists a set of primes~$T:=\{p_{ij}, p_v\}_{(i,j)\in \bA, v\in V'}$ in~$K$ with Frobenius elements~$\sigma_{ij}$ (resp.~$\sigma_v$) in~$F$ conjugated to an element~$l_{ij}$ (resp.~$l_v$) in~$F$ satisfying~$l_{ij}\equiv [x_i;x_j]\pmod{F_3}$ (resp.~$l_v\equiv x_v \pmod{F_3}$. Define~$R_{\bA}$ the normal closed subgroup of~$G_K$ generated by~$l_{\bA}$ and~$l_{V'}$, then we infer~$G_K^T=G_K/R_{\bA}$, which is mild presented by generators~$\{x_1;\dots;x_d\}$ and set of relations~$l_{\bA}$ satisfying the Condition~\eqref{condgradgroup}. Introduce~$K_p^T$ the maximal Galois subextension of~$K_p$ which totally splits in~$T$. 

Define~$R_{\bB}$ the closed normal subgroup of~$G_K^T$ generated by images of~$l_{\bB}:=\{l_{uv}:=[x_u;x_v]; (u,v)\in \bB\}$, and~$K(\Gamma)$ the fixed subfield of~$K_p^T$ by~$R_{\bB}$. Then a presentation of~$G:=\Gal(K(\Gamma)/K)$ is given by~$F/R$, where~$F$ is the free pro-$p$ group generated by~$\{x_1;\dots;x_d\}$ and~$R$ is the closed normal subgroup of~$F$ generated by the family~$l_{\bE}$. 


Since~$l_{\bE}$ satisfies the Condition \eqref{condgradgroup}, using Theorem~\ref{gradgroup}, we infer that~$$\E(G)\simeq \E(\Gamma).$$
Using Proposition~\ref{compcoho}, we conclude that:
$$H^\bullet(G)\simeq \A(\Gamma).$$
\end{proof}

\begin{coro}\label{final coro}
For every integer~$d\geq 5$, there exists a totally imaginary field~$K$ and a non-trivial set~$T$ of primes in~$K$ such that the Galois group~$G_K^T$:
\begin{itemize}
\item[$(i)$] is mild on~$d$ generators,
\item[$(ii)$] admits a non-analytic quotient~$G$ with~$d$ generators and cohomological dimension~$d-2$.
\end{itemize}
\end{coro}

\begin{proof}
We fix~$\Gamma_{\bA}$ a graph with~$2$ vertices~$\{1;2\}$ and one edge:~$\{(1,2)\}$. The graph~$\Gamma_{\bA}$ is bipartite.

If~$d\geq 5$, take~$\Gamma:=\Gamma_{\bA}\sqcup \Gamma_{\bB}$, where~$\Gamma_{\bB}$ is the complete graph on~$d-2$ vertices. The graph~$\Gamma$ satisfies the Condition~\eqref{condgradgroup}, has clique number~$d-2$, and admits~$d$ vertices. Then~$\E(\Gamma)$ has cohomological dimension~$d-2$ and admits~$d$ generators. We conclude with Theorem~\ref{prescribed cohomological dimension2}.
\end{proof}

\section*{Data availability statement}
Data sharing is not applicable to this article as no datasets were generated or
analysed.

\bibliography{bibactbib}
\bibliographystyle{plain}
\end{document}